\documentclass[11pt,a4paper]{article}
\usepackage{amsmath,amssymb,latexsym,amsthm}  
\usepackage[english]{babel}
\usepackage[latin2]{inputenc}

\newtheorem{theorem}{Theorem}

\newtheorem{lemma}[theorem]{Lemma}

\begin{document}

\title{Three-dimensional loops as sections in a four-dimensional solvable Lie group}
\author{\'Agota Figula} 
\date{}
\maketitle

\begin{abstract} We classify all three-dimensional connected topological loops such that the group topologically generated by their left translations is the four-dimensional connected Lie group $G$ which has trivial center and precisely two one-dimensional normal subgroups. We show that $G$ is not the multiplication group of connected topological proper loops.  
\end{abstract}


\noindent
{\footnotesize {2010 {\em Keywords:} Topological loops, sharply transitive sections in groups, multiplication group of loops, solvable Lie groups.}}

\noindent
{\footnotesize {\em Acknowledgments:} Sincere thanks to the referee for the careful reading of the manuscript and for many suggestions. This paper was supported by the Hungarian Scientific Research Fund (OTKA) Grant PD 77392 and by the EEA and Norway Grants (Zolt\'an Magyary Higher Education Public Foundation). }

\section{Introduction} 
A loop $(L, \cdot )$ is a quasigroup with identity element $e \in L$. The left translations $\lambda _a: x \mapsto a x$ and the right translations 
$\rho _a: x \mapsto x a$, $a \in L$, generate a permutation group on the set $L$. This group is called the multiplication group $Mult(L)$ of $L$ (cf. \cite{bruck}). The subgroup $G$ of $Mult(L)$ generated by all left translations of $L$ is the group of left translations of $L$.  The stabilizer $Inn(L)$ of $e \in L$ in the group $Mult(L)$ is called the inner mapping group of $L$. The multiplication group and the inner mapping group of $L$ are essential tools for investigating the structure of the loop $L$. 
An important problem is to be analyzed under which circumstances a group is the multiplication group of a loop. The answer to this question was given in \cite{kepka} which says that we can use certain conditions for transversals. These conditions were applied to study the structures of the multiplication groups and the inner mapping groups of finite loops (cf. \cite{niem3}, \cite{niem4}, \cite{niem1}, \cite{niem2}, \cite{vesanen}).     

Another relevant problem is the classification of loops $L$ having a group $G$ as the group of the left translations of $L$. This problem is equivalent to the following problem in group theory:  to find the subgroups $H$ of $G$ such that the core of $H$ in $G$ is trivial and after this to determine all sections 
$\sigma : G/H \to G$ such that $\sigma (H)=1 \in G$, $\sigma (G/H)$ generates $G$ and acts sharply transitively on the left cosets $x H$, $x \in G$.  The last property means that 
for given cosets $g_1 H$, $g_2 H$ there exists precisely one $z \in \sigma (G/H)$ which satisfies the equation $z g_1 H=g_2 H$. In particular if
 $G$ is a connected Lie group and the section $\sigma $ is continuous with the above property (it is called continuous sharply transitive  section), then we obtain all connected topological loops $L$ such that $G$ is the group of left translations of $L$ (cf. \cite{loops}, Section 1).  
In this paper we investigate connected topological loops. 
 
Concrete classifications of $2$-dimensional connected topological loops $L$ having a Lie group of dimension $3$ as the group of left translations of $L$ are given  in  \cite{loops}, Section 22 and 23, and in \cite{figula0}. The Lie groups which are multiplication groups of $2$-dimensional topological loops are determined in \cite{figula00}. In \cite{figula} we classify all $3$-dimensional simply connected topological loops $L$ having a $4$-dimensional nilpotent Lie group as the group of left translations of $L$. Moreover, we prove that these groups are not multiplication groups of $3$-dimensional topological loops. 

In recent research we deal with the class of solvable Lie groups and study which groups in this class occur as the group of left translations respectively as the multiplication group of $3$-dimensional topological loops. The solvable Lie groups with non-trivial centre play an important role in the investigation of the multiplication group of $3$-dimensional loops (cf. \cite{figula2}). In contrast to this here we consider the $4$-dimensional solvable  Lie group $G$ which has trivial centre and precisely two $1$-dimensional normal subgroups. It depends on a real parameter $a \neq 0$. Its Lie algebra  has two kinds of paracomplex structures (cf. \cite{andrada}, Section 2) and admits only for $a=1$ an invariant complex structure (cf. \cite{ovando}, Theorem 5, pp. 24-25).      

In this paper we classify the $3$-dimensional connected topological loops $L$ which are sections in the Lie group $G$.   
These loops $L$ are homeomorphic to $\mathbb R^3$ and their multiplications depend on one continuous real function of two or three variables 
(cf. Theorem \ref{Propelso}). We give an easy proof of the fact that $G$ is not the multiplication group of connected topological proper loops 
(cf. Theorem \ref{Propmasodik}). 
Later we intend to extend our research to other $4$-dimensional Lie groups.

\section{Preliminaries} 
A binary system $(L, \cdot )$ is called a loop if there exists an element
$e \in L$ such that $x=e \cdot x=x \cdot e$ holds for all $x \in L$ and the
equations $a \cdot y=b$ and $x \cdot a=b$ have precisely one solution, which we denote by $y=a \backslash b$ and $x=b/a$. A loop $L$ is proper if it is not a group.

The left and right translations $\lambda _a: y \mapsto a \cdot y:L \times L \to L$ and $\rho _a: y \mapsto y \cdot a: L \times L \to L$, $a \in L$, are bijections of $L$. The permutation group $Mult(L)$ generated by all left and right translations of the loop $L$ is called the multiplication group of $L$ and the stabilizer of $e \in L$ in the group $Mult(L)$ is called the inner mapping group $Inn(L)$ of $L$. 

Let $K$ be a group, let $S \le K$. Denote by $C_K(S)$ the core of $S$ in $K$ (the largest normal subgroup of $K$ contained in $S$).  A group $K$ is isomorphic to the multiplication group of a loop if and only if there exists a subgroup $S$ with $C_K(S)=1$ and two left transversals $A$ and $B$ to $S$ in $K$ such that $a^{-1} b^{-1} a b \in S$ for every $a \in A$ and $b \in B$ and $K=\langle A, B \rangle $ (cf. Theorem 4.1 in \cite{kepka}, p. 118). Proposition 2.7 in \cite{kepka}, p. 114, yields the following

\begin{lemma} \label{niemenmaa} Let $L$ be a loop with multiplication group $Mult(L)$ and inner mapping group $Inn(L)$. Then 
the normalizer $N_{Mult(L)}(Inn(L))$ is the direct product $Inn(L) \times Z(Mult(L))$, where $Z(Mult(L))$ is the center of the group $Mult(L)$. 
\end{lemma}

A loop $L$ is called topological  if $L$ is a topological space and the binary operations
$(x,y) \mapsto x \cdot y, \ (x,y) \mapsto x \backslash y, (x,y) \mapsto y/x :L \times L \to L$  are continuous. 
Let $G$ be a connected Lie group, $H$ be a subgroup of $G$.  Let $\sigma :G/H \to G$ be a continuous section with respect to the natural projection 
$G \to G/H$.  This is called continuous sharply transitive section, if the set $\sigma (G/H)$ operates sharply transitively on $G/H$, which means that to any $g_1 H$ and $g_2 H$ there exists precisely one $z \in \sigma (G/H)$ with 
$z x H= y H$.     
Every connected topological loop having a Lie group $G$ as the group topologically generated by its left translations is isomorphic to a loop $L$ realized on the factor space $G/H$, where $H$ is a closed subgroup of $G$ with $C_G(H)=1$ and $\sigma :G/H \to G$ is  a continuous sharply transitive section with $\sigma (H)=1 \in G$ such that the subset 
$\sigma (G/H)$ generates $G$. 
The multiplication  of  $L$ on the manifold  $G/H$  is  defined by
$x H \ast y H=\sigma (x H) y H$. Moreover, the subgroup $H$ is the stabilizer of the identity element
$e \in L$ in the group $G$. 

\noindent 
Since there does not exist a multiplication with identity on the sphere $S^2$ a simply connected $3$-dimensional topological loop $L$ having a Lie group $G$ as the group topologically generated by its left translations is homeomorphic either to $\mathbb R^3$ or to $S^3$ (see \cite{gorbatsevich}, p. 210). If the group $G$ is solvable, then the simply connected loop $L$ must be homeomorphic to $\mathbb R^3$ because of the sphere $S^3$ is not a solvmanifold (cf. Theorem 3.2 in \cite{gorbatsevich}, p. 208).  

\noindent
The kernel of a homomorphism $\alpha :(L, \cdot ) \to (L', \ast )$ of a loop $L$ into a loop $L'$ is a normal subloop $N$ of $L$, i.e. a subloop of $L$ such that
\[ x \cdot N=N \cdot x, \ \ (x \cdot N) \cdot y= x \cdot (N \cdot y), \ \ x \cdot ( y \cdot N)=(x \cdot y) \cdot N. \] 
A loop $L$ is solvable if it has a series $1=L_0 \le L_1 \le \cdots \le L_n=L$, where $L_{i-1}$ is normal in $L_i$ and 
$L_i/L_{i-1}$ is an abelian group, $i=1, \cdots ,n$.

\section{Three-dimensional topological loops as sections\\
in a $4$-dimensional solvable Lie group} 
A list of the $4$-dimensional indecomposable Lie algebras can be found in \cite{patera}, Table I, p. 988. All these Lie algebras are solvable. 
Now we consider the Lie algebra $A_{4,2}^a$, $a \neq 0$, in Table I in \cite{patera}. This Lie algebra has a codimension 
$1$ abelian ideal and it has 
a representation as subalgebra of ${\bf gl}_4(\mathbb R)$.
In this section we classify the $3$-dimensional connected topological loops such that the Lie algebra of the group $G$ topologically generated by their left translations is $A_{4,2}^a$, $a \neq 0$. The multiplication of the loops in this class depends on a continuous real function of two or three variables. We prove that the group $G$ cannot be the multiplication group of connected topological proper loops.  
We often use the following lemma.  

\begin{lemma} \label{functional} Let $f: \mathbb R \to \mathbb R$ be a continuous function such that for all $z_1, z_2 \in \mathbb R$ we have 
\begin{equation} \label{equegyenlet} f(z_2)+ e^{-z_2} f(z_1) = f(z_1+z_2). \end{equation} 
Then $f(z)=K(1-e^{-z})$, where $K$ is a real constant.  
\end{lemma}
\begin{proof}
Interchanging $z_1$ and $z_2$ in (\ref{equegyenlet}) we get $f(z_1)+ e^{-z_1} f(z_2) = f(z_1+z_2)$. The right hand side of the last equation is equal to the right hand side of equation (\ref{equegyenlet}). Hence for $z_1 z_2 \neq 0$  one has $\frac{f(z_1)}{1-e^{-z_1}} = \frac{f(z_2)}{1-e^{-z_2}}= K$ 
for a suitable constant $K \in \mathbb R$ and the assertion follows.  
\end{proof}

\begin{theorem} \label{Propelso} Let $G$ be the four-dimensional connected solvable Lie group with trivial center having precisely two one-dimensional normal subgroups and choose for $G$ the representation 
\[ G= \left\{ g(x_1,x_2,x_3,x_4)= \left( \begin{array}{cccc}
e^{a x_4} & 0 & 0 & x_1 \\
0 & e^{x_4} & x_4 e^{x_4} & x_2 \\
0 & 0 &  e^{x_4} & x_3 \\
0 & 0 & 0 & 1 \end{array} \right), x_i \in \mathbb R, i=1,2,3,4 \right\}  \]
with fixed number $a \in \mathbb R \setminus \{ 0 \}$. Let $H$ be a one-dimensional subgroup of $G$ which is not normal in $G$. 
If $H$ is not contained in the commutator subgroup $G'$ of $G$, then using automorphism of $G$ we may choose 
$H= \{ g(0,0,0,x_4); x_4 \in \mathbb R \}$. There does not exist connected topological loop $L$ having $G$ as the group topologically generated by the left translations of $L$ and $H$ as the stabilizer of $e \in L$.  
\newline
\noindent 
If $H$ is contained in the commutator subgroup of $G$ and $a \neq 1$, then using automorphisms of $G$ we may choose $H$ as one of the following subgroups:
\[ H_1= \{ g(0,0,x_3,0); x_3 \in \mathbb R \}, \ H_2= \{ g(x_1,0,x_1,0); x_1 \in \mathbb R \}, \]
\[ H_3=\{ g(x_1,x_1,0,0); x_1 \in \mathbb R \}. \]
For $a=1$ using automorphisms of $G$ we may assume that $H=H_1$.  
\newline
\noindent
a) Every continuous sharply transitive section $\sigma : G/H_1 \to G$ with the properties that $\sigma (G/H_1)$ generates $G$ and 
$\sigma (H_1)=1$ is given by the map $\sigma _f: g(x,y,0,z) H_1 \mapsto g(x,y+z e^{z} f(x,z), e^{z} f(x,z),z)$, 
where $f: \mathbb R^2 \to \mathbb R$ is a continuous function with $f(0,0)=0$ such that the function $f$ does not fulfill the identities $f(x,0)=0$ and $f(0,z)=K(1-e^{-z})$, $K \in \mathbb R$, simultaneously. 
The multiplication of the loop $L_f$ corresponding to $\sigma _f$ can be written as 
\begin{equation} \label{elsoszorzas} (x_1,y_1,z_1) \ast (x_2,y_2,z_2)=(x_1+ e^{a z_1} x_2, y_1+y_2 e^{z_1}-z_2 e^{z_1} f(x_1,z_1), z_1+z_2). \end{equation}  
Every loop $L_f$ defined by (\ref{elsoszorzas}) is solvable. 
\newline
\noindent
b) Each continuous sharply transitive section $\sigma : G/H_2 \to G$ such that $\sigma (G/H_2)$ generates $G$ and 
$\sigma (H_2)=1$ has the form  
\[ \sigma _h: g(x,y,0,z) H_2 \mapsto g(x+e^{az} h(x,y,z),y+z e^{z} h(x,y,z), e^z h(x,y,z),z), \]  
$a \neq 1$, where $h: \mathbb R^3 \to \mathbb R$ is a continuous function with $h(0,0,0)=0$ such that $h$ does not satisfy the identities $h(x,y,0)=0$ and $h(0,0,z)=K(1-e^{-z})$, $K \in \mathbb R$, simultaneously and such that 
for all triples $(x_1,y_1,z_1)$ and $(x_2,y_2,z_2) \in \mathbb R^3$ the equations 
\begin{equation} \label{masodikequ} y= y_2- e^{z_2-z_1} y_1+ e^{z_2-z_1} z_1 h(x,y,z_2-z_1), \end{equation} 
\begin{equation} \label{harmadikequ} x= x_2 -x_1 e^{a(z_2-z_1)}+ e^{a z_2} (e^{-z_1}-e^{-a z_1}) h(x,y,z_2-z_1) \end{equation}
have a unique solution $(x,y) \in \mathbb R^2$. 
The multiplication of the loop $L_h$ corresponding to $\sigma _h$ is determined by 
\begin{equation} (x_1,y_1,z_1) \ast (x_2,y_2,z_2)= \nonumber \end{equation}
\begin{equation} \label{multiplication1} \big(x_1+ e^{a z_1}(x_2 + h(x_1,y_1,z_1)[1-e^{(a-1)z_2}]), y_1+ e^{z_1}(y_2 -z_2 h(x_1,y_1,z_1)),z_1+z_2 \big). \end{equation}  
\newline
\noindent
c) Every continuous sharply transitive section $\sigma : G/H_3 \to G$ with the properties $\sigma (G/H_3)$ generates $G$ and 
$\sigma (H_3)=1$ is given by the map 
\[ \sigma _f: g(x,0,y,z) H_3 \mapsto g(x+e^{az} f(x,y,z),e^z f(x,y,z),y,z), \ a \neq 1, \] 
where $f: \mathbb R^3 \to \mathbb R $ is a continuous function with $f(0,0,0)=0$ such that $f$ does not fulfill the identities $f(x,y,0)=-x$ and $f(0,0,z)= c(1 -e^{-a z})$, $c \in \mathbb R$, simultaneously and such that for all triples $(x_1,y_1,z_1)$, 
$(x_2,y_2,z_2) \in \mathbb R^3$ there is a unique $x \in \mathbb R$ satisfying the equation 
\begin{equation}  \label{negyedikujequ} x= x_2- x_1 e^{a (z_2-z_1)}+ \nonumber \end{equation} 
\begin{equation}  e^{a z_2-z_1}[y_1 (z_2-z_1)+ (1- e^{(1-a) z_1}) f(x,y_2-e^{z_2-z_1} y_1,z_2-z_1)]. \end{equation} 
The multiplication of the loop $L_f$ corresponding to $\sigma _f$ is defined by 
\begin{equation} \label{loopszorzasketto} (x_1,y_1,z_1) \ast (x_2,y_2,z_2)= \nonumber \end{equation}
\begin{equation} \big(x_1+e^{a z_1}(x_2-y_2 z_1 e^{(a-1)z_2}+f(x_1,y_1,z_1)[1-e^{(a-1)z_2}]), y_1+ e^{z_1} y_2, z_1+z_2 \big). \end{equation}
\end{theorem} 
\begin{proof}  
The linear representation of the Lie group $G$ is given in \cite{ghanam}, p. 164. 
The Lie algebra ${\bf g}$ of $G$ is given by the basis $\{ e_1,e_2,e_3,e_4 \}$ with $[e_1,e_4]=a e_1$, $[e_2,e_4]=e_2$, $[e_3,e_4]=e_2+e_3$, 
$a \neq 0$.  As $\mathbb R e_1$ and $\mathbb R e_2$ are ideals of ${\bf g}$ the subalgebra ${\bf h}$ of the $1$-dimensional non-normal subgroup $H$ of $G$ does not contain $e_1$, $e_2$. First we assume that ${\bf h}$ is not contained in the commutator subalgebra ${\bf g}'= \langle e_1, e_2, e_3 \rangle $ of ${\bf g}$. Hence one has ${\bf h}= \mathbb R (a e_1+ b e_2+ c e_3+ e_4)$ with $a,b,c \in \mathbb R$. Using the automorphism $e_1 \mapsto e_1$, $e_2 \mapsto e_2$, $e_3 \mapsto e_3$ and $e_4 \mapsto e_4 + a e_1+ b e_2+ c e_3$ of ${\bf g}$ we can assume ${\bf h}= \mathbb R e_4$.  
Now we assume that there is a connected topological loop $L$ having the group $G$ as the group topologically generated by its left translations and 
$H =\{ \exp t e_4; t \in \mathbb R \}$ as the stabilizer of $e \in L$. As every $1$-dimensional subalgebra which is not contained in the commutator subalgebra of the Lie algebra ${\bf g}$ is conjugate to ${\bf h}$ every 
element of $G$ not contained in the commutator subgroup $G'$ of $G$ has a fixed point on $G/H$.  As the set $\sigma (G/H)$ is the set of the left translations of the loop $L$ and the left translations have no fixed point $\sigma (G/H)$ should be contained in $G'$. This is a contradiction to the fact that the left translations of $L$ generate $G$ and the first assertion is proved. 
\newline
\noindent
If the subalgebra ${\bf h}$ of $H$ is contained in the commutator subalgebra ${\bf g}'$, then $H$ has the form 
$H=\exp \ t (b_1 e_3 + b_2 e_1+ b_3 e_2)$, $t \in \mathbb R$, with $b_1 \neq 0$ or $b_2 b_3 \neq 0$. If $a \neq 1$, then each automorphism $\varphi $ of ${\bf g}$ is given by 
$\varphi (e_1)= k e_1$, $\varphi (e_2)= l e_2$, $\varphi (e_3)= n e_2 + l e_3$, $\varphi (e_4)= f_1 e_1+ f_2 e_2 + f_3 e_3+ e_4$ with $kl \neq 0$, $k,l,n,f_1,f_2,f_3 \in \mathbb R$.  
If $a=1$, then the automorphism group of ${\bf g}$ consists of the linear mappings
$\alpha (e_1)= k_1 e_1 + k_2 e_2$, $\alpha (e_2)= l e_2$, $\alpha (e_3)=n_1 e_1+ n_2 e_2 + l e_3$, $\alpha (e_4)= f_1 e_1+ f_2 e_2 + f_3 e_3+ e_4$ with $k_1 l \neq 0$, $k_1,k_2,l,n_1,n_2,f_1,f_2,f_3 \in \mathbb R$. If $b_1 \neq 0$ and $a \neq 1$, then we can change $H$ by an automorphism 
$\varphi $ of $G$ such that $H$ has one of the following forms  
\[H_1=\{ \exp t e_3; t \in \mathbb R \}, \quad H_2= \{ \exp t(e_3+e_1); t \in \mathbb R \}. \]
If $b_1 \neq 0$ and $a=1$, then using an automorphism $\alpha $ of $G$ we may assume that $H=H_1$. If $b_1=0$ and $a \neq 1$, then we can change $H$ by an automorphism $\varphi $ of $G$ such that $H$ is the subgroup 
$H_3= \{ \exp t(e_1+e_2); t \in \mathbb R \}$. If $b_1=0$ and $a=1$, then for all $b_2, b_3 \in \mathbb R$ the subgroup 
$H= \{ \exp \ t (b_2 e_1+ b_3 e_2); t \in \mathbb R \}$ is normal in $G$ which is a contradiction. 

First we deal with the case that $H=H_1=\{ g(0,0,k,0); \ k \in \mathbb R \}$. Since all elements of $G$ have a unique decomposition as 
$g(x, y, 0, z)  g(0, 0, k, 0)$, any continuous function
$f: \mathbb R^3 \to \mathbb R; (x,y, z) \mapsto f(x,y,z)$ determines a continuous section $\sigma : G/H_1 \to G$ given by
\[ \sigma: g(x,y,0,z) H_1 \mapsto  g(x, y, 0, z) g(0, 0, f(x,y,z), 0) = \]
\[ g(x,y+z e^z f(x,y,z), e^z f(x,y,z),z). \] 
The section $\sigma $ is sharply transitive if and only if for every triples $(x_1,y_1,z_1)$, $(x_2,y_2,z_2) \in \mathbb R^3$ there exists precisely one triple  $(x,y,z) \in \mathbb R^3$ such that 
\[ g(x,y+z e^z f(x,y,z), e^z f(x,y,z), z) g(x_1, y_1, 0, z_1)= g(x_2, y_2, 0, z_2) g(0,0,t,0) \] 
holds with a suitable $t \in \mathbb R$. This gives the equations
\[ z= z_2-z_1, \ x= x_2-x_1 e^{a (z_2-z_1)}, \ t= e^{-z_1} f(x_2-x_1 e^{a (z_2-z_1)},y,z_2-z_1), \]
\begin{equation} \label{equelso} 0= y -y_2 +y_1 e^{z_2-z_1}- z_1 e^{z_2-z_1} f(x_2-x_1 e^{a (z_2-z_1)},y,z_2-z_1). \nonumber \end{equation}
These are equivalent to the condition that for every $z_0=z_2-z_1$, $x_0= x_2-x_1 e^{a z_0}$ and $z_1 \in \mathbb R$ the function 
$g: y \mapsto y - z_1 e^{z_0} f(x_0, y, z_0): \mathbb R \to \mathbb R$ is a bijective mapping. Let be $\psi_1 < \psi_2 \in \mathbb R$ then 
$g(\psi_1) \neq g(\psi_2)$, e.g. $g(\psi _1) < g(\psi _2)$. We consider 
\[ 0 <  g(\psi _2) - g(\psi _1)= \psi _2 - \psi _1 - z_1 e^{z_0}[f(x_0, \psi_2, z_0) - f(x_0, \psi_1, z_0)] \]
as a linear function of $z_1 \in \mathbb R$. If $f(x_0, \psi_2, z_0) \neq f(x_0, \psi_1, z_0)$, then there exists a $z_1 \in \mathbb R$ 
such that $g(\psi _2) - g(\psi _1) =0$, which is a contradiction. Hence the function $f(x, y, z)=f(x,z)$ does not depend on $y$. In this case $g$ is a monotone function and every continuous function $f(x,z)$ with $f(0,0)=0$ determines a loop multiplication. 
\newline 
\noindent
This loop is proper precisely if the set 
\[ \sigma (G/H_1)=\{ g(x,y+z e^z f(x,z), e^z f(x,z),z); \\ x,y,z \in \mathbb R \}  \]
generates the whole group $G$.  
The set $\sigma (G/H_1)$ contains the subgroup  
\[ G_1=\{ g(x, y, f(x,0), 0); \ x,y \in \mathbb R \} < G'=[G,G], \]
and the subset $F_2=\{ g(0,z e^z f(0,z), e^z f(0,z),z); \ z \in \mathbb R \}$.  
As $G_1 \cap F_2= \{ 1 \}$ the set $\sigma (G/H_1)$ generates $G$ if the group $G_1$ has dimension $3$. This is the case if and only if the subgroup $G_2=\{ g(x,0,f(x,0),0); x \in \mathbb R \}$ is not a one-parameter subgroup. But $G_2$ is a one-parameter subgroup precisely if $f(x,0)= \lambda x$, $\lambda \in \mathbb R$. In this case the set $\sigma (G/H_1)$ generates $G$ if there exists an element $h \in F_2$ with 
$h^{-1} G_1 h \neq G_1$. For $h= g(0,z e^z f(0,z), e^z f(0,z),z) \in F_2$, where $z \neq 0$, we have 
\[ h^{-1} g(x, y, \lambda x, 0) h= g(x e^{-a z}, y e^{-z}-z e^{-z} \lambda x, e^{-z} \lambda x,0). \]
Hence $h^{-1} G_1 h =G_1$ if and only if for $a \neq 1$ one has $f(x,0)=0$ and for $a=1$ we have $f(x,0)=\lambda x$,  
$\lambda \in \mathbb R$. Using this for $a \neq 1$ the group $G_1$ reduces to $\widetilde{G_1}=\{ g(x,y,0,0); x,y \in \mathbb R \}$ and for $a=1$  
to  $G_1^{\ast }=\{ g(x,y, \lambda x,0); x,y \in \mathbb R \}$, $\lambda \in \mathbb R$. The group $\widetilde{G_1}$ and for $a=1$ the group  $G_1^{\ast }$ are normal subgroups of $G$. The set $\sigma (G/H)$ does not generate $G$ precisely if for $a \neq 1$ 
the set $F_2 \widetilde{G_1}/\widetilde{G_1}$, respectively for $a=1$ the set $F_2 G_1^{\ast }/G_1^{\ast }$ is a one-parameter subgroup of $G/\widetilde{G_1}$, respectively $G/G_1^{\ast }$. Since for $a \neq 1$ one has 
\[g(\mathbb R, \mathbb R, e^{z_1} f(0,z_1),z_1) g(\mathbb R, \mathbb R, e^{z_2} f(0,z_2),z_2)= \] 
\[g(\mathbb R, \mathbb R, e^{z_1+z_2} f(0,z_2)+e^{z_1} f(0,z_1), z_1+z_2) \]
and for $a=1$ we have 
\[g(\mathbb R, \mathbb R, e^{z_1}( \lambda x_1 +f(0,z_1)),z_1) g(\mathbb R, \mathbb R, e^{z_2}( \lambda x_2+ f(0,z_2)),z_2)= \] 
\[g(\mathbb R, \mathbb R, e^{z_1}[e^{z_2}(\lambda x_2+ f(0,z_2))+ \lambda x_1+f(0,z_1)], z_1+z_2) \]
the set $\sigma (G/H_1)$ does not generate $G$ if and only if in both cases $f(x,0)=0$ and for all $z_1, z_2 \in \mathbb R$ one has  $f(0,z_2)+e^{-z_2} f(0,z_1)=f(0,z_1+z_2)$.   
Using Lemma \ref{functional} for the real function $f(0,z)$ we obtain 
$f(0,z)= K(1- e^{-z})$, where $K \in \mathbb R$. 
Hence the set $\sigma (G/H_1)$ does not generate $G$ if for all $x,z \in \mathbb R$ one has $f(x,0)=0$ and $f(0,z)=K(1- e^{-z})$, $K \in \mathbb R$. 

Now we represent the loop $L_f$ in the coordinate system $(x,y,z) \mapsto g(x,y,0,z)H_1$.  Then the product $(x_1,y_1,z_1) \ast (x_2,y_2,z_2)$ will be determined if we apply $\sigma (g(x_1,y_1,0,z_1)H_1)=g(x_1,y_1+z_1 e^{z_1} f(x_1,z_1), e^{z_1} f(x_1,z_1),z_1)$ 
to the left coset 
$g(x_2,y_2,0,z_2)H_1$ and find in the image coset the element of $G$ which lies in the set $\{ g(x,y,0,z)H_1; \ x,y,z \in \mathbb R \}$. A direct computation gives multiplication (\ref{elsoszorzas}) in assertion a). 

The set $N=\{ (x,y,0); x,y \in \mathbb R \}$ is a normal subgroup isomorphic to $\mathbb R^2$ of the loop $L_f$ and one has $(0,0,z_1)N \ast (0,0,z_2)N= (0,0,z_1+z_2)N$. Hence the factor loop $L_f/N$ is isomorphic to the Lie group $\mathbb R$ and therefore the loop $L_f$ is solvable.

\noindent
Now we assume that $H=H_2=\{ g(k,0,k,0); \ k \in \mathbb R \}$. As all elements of $G$ can be written in a unique way  as 
$g(x, y, 0, z)  g(k, 0, k, 0)$, every continuous function
$h: \mathbb R^3 \to \mathbb R; (x,y, z) \mapsto h(x,y,z)$ determines a continuous section $\sigma : G/H_2 \to G$ defined by
\begin{equation} \label{section} \sigma_h: g(x,y,0,z) H_2 \mapsto  g(x, y, 0, z) g(h(x,y,z), 0, h(x,y,z), 0) = \nonumber \end{equation} 
\begin{equation}  g(x+e^{az} h(x,y,z),y+z e^{z} h(x,y,z), e^z h(x,y,z),z). \end{equation}  
The section $\sigma $ is sharply transitive if and only if for each triples $(x_1,y_1,z_1)$, $(x_2,y_2,z_2) \in \mathbb R^3$ there exists precisely one triple  $(x,y,z) \in \mathbb R^3$ such that 
\begin{equation} \label{elsoequ} g(x+e^{az} h(x,y,z),y+z e^{z} h(x,y,z), e^z h(x,y,z),z) g(x_1, y_1, 0, z_1)= \nonumber \end{equation} 
\begin{equation} g(x_2, y_2, 0, z_2) g(0,0,t,0) \end{equation}  
for a suitable $t \in \mathbb R$. Equation (\ref{elsoequ}) gives $z= z_2-z_1$, $t= e^{-z_1} h(x,y,z_2-z_1)$ and that equations (\ref{masodikequ}),   
(\ref{harmadikequ}) in assertion b) must have a unique solution $(x,y) \in \mathbb R^2$.  
\newline
\noindent 
Now we investigate under which circumstances the set $\sigma (G/H_2)$ generates the group $G$.  
The set $\sigma (G/H_2)$ contains the subgroup  
\[ G_1=\{ g(x+h(x,y,0), y, h(x,y,0), 0); \ x,y \in \mathbb R \} < G'=[G,G], \]
and the subset $F_2=\{ g(e^{az} h(0,0,z),z e^z h(0,0,z), e^z h(0,0,z),z); \ z \in \mathbb R \}$.  
One has $G_1 \cap F_2= \{ 1 \}$. The set $\sigma (G/H_2)$ generates $G$ if $\hbox{dim} \ G_1=3$. This happens if 
the group $G_2=\{ g(h(0,y,0), y, h(0,y,0), 0); y \in \mathbb R \}$ 
or the group $G_3=\{ g(x+h(x,0,0), 0, h(x,0,0), 0); x \in \mathbb R \}$ 
has dimension $2$. The group $G_2$, respectively $G_3$ is a one-parameter subgroup precisely if $h(0,y,0)=c y$, $c \in \mathbb R$, respectively 
$h(x,0,0)=b x$, $b \in \mathbb R$. As 
\[ g(x+h(x,0,0), 0, h(x,0,0), 0) g(h(0,y,0), y, h(0,y,0), 0)= \] 
\[ g(x+h(x,0,0)+h(0,y,0), y, h(x,0,0)+h(0,y,0), 0) \]
the group $G_1$ has dimension $3$ if and only if the function $h(x,y,0)$ is different from $b x+c y$, $b,c \in \mathbb R$. 
Assuming this the set $\sigma (G/H_2)$ generates $G$ if there exists an element $h \in F_2$ such that 
$h^{-1} G_1 h \neq G_1$ holds. For $h= g(e^{az} h(0,0,z),z e^z h(0,0,z), e^z h(0,0,z),z) \in F_2$ with $z \neq 0$ one has 
\[ h^{-1} g(x+b x+c y, y, b x+c y, 0) h= \]
\[ g([(b+1)x+cy] e^{-a z}, y e^{-z}-z (b x+c y) e^{-z}, (b x+c y) e^{-z},0). \] 
We obtain $h^{-1} G_1 h =G_1$ if and only if $b=c=0$. Then the group $G_1=\{ g(x, y, 0, 0); \ x,y \in \mathbb R \}$ is a normal subgroup of $G$.   
The set $\sigma (G/H_2)$ generates $G$, if the set 
$(F_2 G_1)/G_1$ is not a one-parameter subgroup of $G/G_1$. Because of
\begin{equation} g(\mathbb R, \mathbb R, e^{z_1} h(0,0,z_1), z_1) g(\mathbb R , \mathbb R, e^{z_2} h(0,0,z_2), z_2)= \nonumber \end{equation}
\begin{equation} g(\mathbb R, \mathbb R, e^{z_1+z_2} h(0,0,z_2)+e^{z_1} h(0,0,z_1),z_1+z_2), \nonumber \end{equation} 
the set $\sigma (G/H_2)$ does not generate $G$ precisely if for all $z_1, z_2 \in \mathbb R$ the equality 
$h(0,0,z_2)+e^{-z_2} h(0,0,z_1)= h(0,0,z_1+z_2)$ 
holds. Using Lemma \ref{functional} for the function $h(0,0,z)$ we get $h(0,0,z)= K(1 -e^{-z})$ with $K \in \mathbb R$. Therefore the loop $L_h$ 
is proper if and only if the function $h: \mathbb R^3 \to \mathbb R$ does not satisfy the identities $h(x,y,0)=0$ and $h(0,0,z)= K(1 -e^{-z})$ 
simultaneously. 
\newline
\noindent
The multiplication of the loop $L_h$  in the coordinate system
$(x,y,z) \mapsto g(x,y,0,z)H_2$ is determined if we apply 
$\sigma _h(g(x_1,y_1,0,z_1) H_2)$ given by (\ref{section}) 
 to the left coset $g(x_2,y_2,0,z_2)H_2$ and find in the image coset the element of $G$ which lies in the set 
 $\{g(x,y,0,z)H_2; \ x,y,z \in \mathbb R \}$.
A direct computation yields multiplication (\ref{multiplication1}) of assertion b). 

\noindent
Now we consider the case that $H=H_3=\{ g(k,k,0,0); \ k \in \mathbb R \}$. As all elements of $G$ have a unique decomposition as 
$g(x,0,y,z)  g(k,k,0,0)$, any continuous function
$f: \mathbb R^3 \to \mathbb R; (x,y, z) \mapsto f(x,y,z)$ determines a continuous section $\sigma : G/H_3 \to G$ given by
\[ \sigma : g(x,0,y,z) H_3 \mapsto  g(x,0,y,z) g(f(x,y,z), f(x,y,z), 0, 0) = \]
\[ g(x+e^{az} f(x,y,z),e^z f(x,y,z),y,z). \] 
The set $\sigma (G/H_3)$ acts sharply transitively on the factor space $G/H_3$ if and only if for every triples $(x_1,y_1,z_1)$, $(x_2,y_2,z_2) \in \mathbb R^3$ there exists precisely one triple  $(x,y,z) \in \mathbb R^3$ such that 
\begin{equation} \label{negyedikequ} g(x+e^{az} f(x,y,z),e^z f(x,y,z),y,z) g(x_1,0,y_1,z_1)= \nonumber \end{equation} 
\begin{equation} g(x_2,0, y_2,z_2) g(t,t,0,0) \end{equation}  
for a suitable $t \in \mathbb R$. Equation (\ref{negyedikequ}) yields that 
$z= z_2-z_1$, $y= y_2-e^{z_2-z_1} y_1$, 
$t= y_1 (z_2-z_1) e^{-z_1}+ e^{-z_1} h(x,y_2-e^{z_2-z_1} y_1,z_2-z_1)$ and equation (\ref{negyedikujequ}) in assertion c) must have a unique solution $x \in \mathbb R$.  
\newline 
\noindent 
The set $\sigma (G/H_3)$ contains the subgroup  
\[ G_1=\{ g(x+f(x,y,0), f(x,y,0), y, 0); \ x,y \in \mathbb R \} < G'=[G,G], \]
and the subset $F_2=\{ g(e^{az} f(0,0,z), e^z f(0,0,z), 0,z); \ z \in \mathbb R \}$.  
We have $G_1 \cap F_2= \{ 1 \}$. The set $\sigma (G/H_3)$ generates $G$ if $G_1$ is a $3$-dimensional group.  This is the case if 
the subgroup 
\[G_2=\{ g(x+f(x,0,0), f(x,0,0), 0, 0); x \in \mathbb R \} \] 
or the subgroup 
\[G_3=\{ g(f(0,y,0), f(0,y,0), y, 0); y \in \mathbb R \} \] 
has dimension $2$. The group $G_2$, respectively $G_3$ is a one-parameter subgroup precisely if $f(x,0,0)=c x$, $c \in \mathbb R$, respectively  $f(0,y,0)=d y$, $d \in \mathbb R$. As 
\[ g(x+f(x,0,0), f(x,0,0), 0, 0) g(f(0,y,0), f(0,y,0), y, 0)= \] 
\[ g(x+f(x,0,0)+f(0,y,0), f(x,0,0)+f(0,y,0), y, 0) \]
we have $\hbox{dim} \ G_1=3$ if and only if the function $f(x,y,0)$ is different from $c x+d y$, $c,d \in \mathbb R$. 
In this case the set $\sigma (G/H_3)$ generates $G$ if there exists an element $h \in F_2$ such that 
$h^{-1} G_1 h \neq G_1$ holds. For $h= g(e^{az} f(0,0,z), e^z f(0,0,z), 0,z) \in F_2$ with $z \neq 0$ we have 
\[ h^{-1} g(x+c x+d y, c x+d y, y, 0) h= \]
\[ g([(c+1)x + dy] e^{-a z}, (-z y + c x+d y)) e^{-z}, e^{-z} y, 0). \] 
Hence $h^{-1} G_1 h =G_1$ if and only if one has $c=-1$ and $d=0$. Then the group $G_1=\{ g(0, -x, y, 0); \ x,y \in \mathbb R \}$ is a normal subgroup of $G$ and the set $\sigma (G/H_3)$ generates $G$, if the set 
$(F_2 G_1)/G_1$ is not a one-parameter subgroup of $G/G_1$. Because of
\begin{equation} g(e^{a z_1} f(0,0,z_1), \mathbb R, \mathbb R,  z_1) g(e^{a z_2} f(0,0,z_2), \mathbb R, \mathbb R,  z_2)= \nonumber \end{equation}
\begin{equation} g(e^{a (z_1+z_2)} f(0,0,z_2)+ e^{a z_1} f(0,0,z_1), \mathbb R, \mathbb R, z_1+z_2), \nonumber \end{equation} 
the set $\sigma (G/H_3)$ does not generate $G$ precisely if for all $z_1, z_2 \in \mathbb R$ the identity 
$f(0,0,z_2)+e^{-a z_2} f(0,0,z_1)= f(0,0,z_1+z_2)$  
holds. Using Lemma \ref{functional} for the function $f(0,0,z)$ we obtain $f(0,0,z)=c (1-e^{-a z})$ with a real constant $c \in \mathbb R$.  
A direct computation yields that the multiplication of the loop $L_f$ corresponding to the section $\sigma _f$ in the coordinate system $(x,y,z) \mapsto g(x,0,y,z) H_3$ is given by (\ref{loopszorzasketto}) and the assertion c) is proved. \end{proof} 

\begin{theorem} \label{Propmasodik} The $4$-dimensional connected solvable Lie group $G$ which has trivial center and precisely two one-dimensional normal subgroups is not the multiplication group of connected topological proper loops. 
\end{theorem}
\begin{proof}  
Every $1$-dimensional connected topological loop $L$ having a Lie group as its multiplication group is already a Lie group (cf. Theorem 18.18 in \cite{loops}, p. 248). If the multiplication group of a $2$-dimensional connected topological proper loop is a Lie group, then it is nilpotent 
(cf. Theorem 1 in \cite{figula00}).  
Since the group $G$ has trivial center all $3$-dimensional connected topological loop $L$ having the group $G$ as the group topologically generated by the left translations of $L$ are  homeomorphic to $\mathbb R^3$ and their multiplications are given by (\ref{elsoszorzas}), (\ref{multiplication1}) and (\ref{loopszorzasketto}) in Theorem \ref{Propelso}. Denote by $H_i$, $i=1,2,3$, the subgroups of $G$ given in Theorem \ref{Propelso}. 
If the group $G$ is also the group generated by all left and right translations of $3$-dimensional topological loops, then the inner mapping group of $L$ is the group $H_1$ if $L$ is defined by (\ref{elsoszorzas}), the group $H_2$ if $L$ is defined by (\ref{multiplication1}) and the group $H_3$ if $L$ is given by (\ref{loopszorzasketto}). Since the commutator subgroup $G'$ of $G$ is a $3$-dimensional abelian normal subgroup of $G$ and each subgroup $H_i$, $i=1,2,3$, is contained in $G'$ the normalizer of $H_i$ in $G$ coincides with the group $G'$. As  the Lie algebra ${\bf g}$ of the group $G$ has trivial center we have a contradiction to Lemma \ref{niemenmaa} and the assertion follows. \end{proof}

Author's address: Institute of Mathematics, University of Debrecen,\\
Debrecen, H-4010, P.O.Box: 12, Hungary \\
E-mail: figula@math.klte.hu

\end{document}